\documentclass{amsart}
\usepackage{graphicx}
\usepackage[T1]{fontenc}
\usepackage{amsmath}
\usepackage{amssymb}
\usepackage{amsthm}
\usepackage{epsfig}
\usepackage{subfig}
\usepackage{color}
\usepackage{hyperref}
\usepackage{enumitem}
\usepackage{listings}
\usepackage{color}
\usepackage{lscape}
\definecolor{gray}{rgb}{0.5,0.5,0.5}
\usepackage{tikz}
\usepackage{pgfplots}
\pgfplotsset{compat=1.18} 
\usepackage{amsmath}

\newlist{steps}{enumerate}{1}
\setlist[steps, 1]{label = Step \arabic*:}

\usepackage[normalem]{ulem}
\input xy
\xyoption{all}
\usepackage[all]{xy} 

\newtheorem{theorem}{Theorem}[section]
\newtheorem{proposition}[theorem]{Proposition}
\newtheorem{lemma}[theorem]{Lemma}

\newtheorem{corollary}[theorem]{Corollary}

\theoremstyle{definition}
\newtheorem{definition}[theorem]{Definition}
\newtheorem{example}[theorem]{Example}

\theoremstyle{remark}
\newtheorem{remark}[theorem]{Remark}

\def\MR#1{\href{http://www.ams.org/mathscinet-getitem?mr=#1}{MR}}

\usepackage{xcolor}
\usepackage{ulem}  
\definecolor{riyared}{rgb}{0,0.5,0} 
\definecolor{riyablue}{rgb}{0,0.5,0.5} 
\newcommand{\rout}{\bgroup\markoverwith{\textcolor{riyared}{\rule[0.5ex]{2pt}{1pt}}}\ULon} 

\DeclareMathOperator{\fix}{Fix}
\DeclareMathOperator{\edeg}{EDeg}

\newcommand{\A}{A}

\newcommand{\Gne}{\mathcal{G}_{n,e}}

\begin{document}

\title[Lights Out on Nearly Complete Graphs]{Lights Out on Nearly Complete Graphs}

\author{Bradley Forrest}
\address{101 Vera King Farris Drive; Galloway, NJ 08205}
\email{bradley.forrest@stockton.edu}

\author{Riya Goyal}
\address{} 
\email{rg85@njit.edu} 



\begin{abstract}
We study the generalization of the game Lights Out in which the standard square grid board is replaced by a graph. We examine the probability that, when a graph is chosen uniformly at random from the set of graphs with $n$ vertices and $e$ edges, the resulting game of Lights Out is universally solvable. Our work focuses on nearly complete graphs, graphs for which $e$ is close to $\binom{n}{2}$. For large values of $n$, we prove that, among nearly complete graphs, the probability of selecting a graph that gives a universally solvable game of Lights Out is maximized when $e = \binom{n}{2} - \lfloor \frac{n}{2} \rfloor$. More specifically, we prove that for any fixed integer $m > 0$, as $n$ approaches $\infty$, this value of $e$ maximizes the probability over all values of $e$ from $\binom{n}{2} - \lfloor \frac{n}{2} \rfloor - m$ to $\binom{n}{2}$.
\end{abstract}
\maketitle

\section{Introduction}
Lights Out is a one-player handheld electronic puzzle game played on a $5 \times 5$ grid of square lights. The player's goal is, when provided with an initial configuration of lights that are {\it on}, to turn off all of the lights. The lights do not toggle individually; each light can be pressed, and pressing a light toggles not only that light but each light that shares a boundary edge with the pressed light.

Generalizations of Lights Out have been the subject of substantial mathematical research. The most common generalization alters the geometry of the lights. Instead of playing Lights Out on a grid, we replace the grid with a graph $\Gamma$. In this game, the lights are the vertices $V\Gamma$ of $\Gamma$, and pressing a light toggles that vertex and each vertex adjacent to it. Lights Out is {\it universally solvable} on $\Gamma$ if, for each initial configuration $c \subseteq V\Gamma$, the resulting puzzle can be solved. In this work, we study which {\it nearly complete graphs}, graphs with $n$ vertices and $e$ edges where $e$ is close to the maximum number $N = \binom{n}{2}$ of possible edges, yield a universally solvable game of Lights Out. While we occasional make use of vertex labels, unless otherwise stated, the graphs in this document are unlabeled.

Many authors have examined universal solvability for this graph theoretic generalization of Lights Out \cite{AS1, AS2, Anderson, CM, CHKR, EES, HY}. In \cite{FM}, Forrest and Manno considered the problem of Lights Out on random graphs and used computational methods to show that when a graph is chosen uniformly at random from the set of graphs with $n$ vertices, the probability that Lights Out is universally solvable on $\Gamma$ is approximately .419 as $n$ approaches $\infty$; some authors have expanded on this work by considering Lights Out on graphs with $n$ vertices \cite{AR, JY}. We consider a related question: when a graph $\Gamma$ is chosen uniformly at random from the set $\Gne$ of graphs with $n$ vertices and $e$ edges, what is the probability $P_{n,e}$ that Lights Out is universally solvable on $\Gamma$? We make use of both deductive and computational methods. As an example of our computational results and as motivation for our investigation, consider Figure \ref{fig:11verts}. This figure gives, for each integer $e$ from 1 to $\binom{11}{2} - 1 = 54$, the approximate probability when a graph $\Gamma$ is chosen uniformly at random from $\mathcal{G}^{11}_e$ that $\Gamma$ will have a universally solvable game of Lights Out. We obtained these approximate probabilities by Monte Carlo experiments which are discussed in detail in Section 8. There are many aspects of Figure \ref{fig:11verts} that merit exploration; in this work we examine the local maximum that occurs when $e = 50 = \binom{11}{2} - \lfloor \frac{11}{2} \rfloor$. We have partially answered two questions regarding this local maximum.
\begin{enumerate}
    \item For what range of $n$ does this local maximum occur?
    \item When this local maximum occurs, over what range of $e$ is it the absolute maximum?
\end{enumerate}
Our main theorem, Theorem \ref{thm:main}, states that for sufficiently large $n$ this local maximum occurs; our computational results suggest that this local maximum occurs for all $n$  $\geq 2$. Additionally, Theorem \ref{thm:main} states that for any fixed positive integer $m > 0$ and $n$ sufficiently large, $e = N - \lfloor \frac{n}{2} \rfloor$ gives that maximum probability over all values of $e$ ranging from $N - \lfloor \frac{n}{2} \rfloor- m$ to $N$.

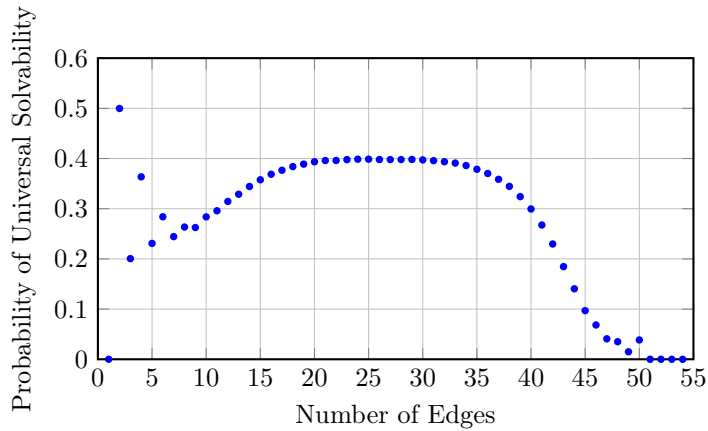
\begin{figure}[htbp]
\centering
\begin{tikzpicture}
\begin{axis}[
    width=0.85\textwidth,
    height=0.5\textwidth,
    xlabel={Number of Edges},
    ylabel={Probability of Universal Solvability},
    ymin=0, ymax=0.6,
    xmin=0, xmax=55,
    xtick={0,5,...,55},
    ytick={0,0.1,...,0.6},
    grid=both,
    grid style={line width=.1pt, draw=gray!30},
    major grid style={line width=.2pt,draw=gray!50},
    thick,
    every axis plot/.append style={only marks, mark=*, mark size=1pt},
]
\addplot[color=blue] coordinates {
    (1,0) (2,0.499724) (3,0.200551) (4,0.363507) (5,0.230845)
    (6,0.283949) (7,0.244384) (8,0.263545) (9,0.262488) (10,0.283742)
    (11,0.295869) (12,0.314470) (13,0.328777) (14,0.344268) (15,0.357644)
    (16,0.368724) (17,0.376659) (18,0.383929) (19,0.388856) (20,0.393536)
    (21,0.396004) (22,0.396182) (23,0.397806) (24,0.398648) (25,0.398657)
    (26,0.397970) (27,0.397904) (28,0.397958) (29,0.398135) (30,0.397270)
    (31,0.395805) (32,0.393628) (33,0.390987) (34,0.386070) (35,0.378623)
    (36,0.370198) (37,0.358598) (38,0.344516) (39,0.324121) (40,0.299494)
    (41,0.267545) (42,0.229736) (43,0.184645) (44,0.140461) (45,0.096906)
    (46,0.068320) (47,0.040787) (48, 0.035003) (49, 0.014937) (50, 0.038504) 
    (51,0) (52,0) (53,0) (54,0)
};
\end{axis}
\end{tikzpicture}
\caption{
  For each value of $e$ from 1 to $\binom{11}{2} - 1 = 54$, the approximate probability that, when a graph is chosen uniformly at random from the set of graphs with 11 vertices and $e$ edges, Lights Out is universally solvable on that graph.
}
\label{fig:11verts}
\end{figure}

This paper is organized as follows. Section 2 introduces the relationship between Lights Out and linear algebra. Section 3 examines dominating sets and uses them to explore the interaction between graph joins and universal solvability. In Section 4, we recast the results of Sections 2 and 3 in terms of complement graphs. Section 5 introduces and applies the excess degree of a graph to study universal solvability in Lights Out. In Section 6, we use the results of Section 5 to determine, for the case where $e$ is slightly less than $N - \lfloor \frac{n}{2} \rfloor$, the number of graphs with $n$ vertices and $e$ edges on which Lights Out is universally solvable. Section 7 proves our main theorem, Theorem \ref{thm:main}, by applying the results of Section 6. Section 8 describes our computational methods and gives the results of our Monte Carlo experiments. In Section 9, we discuss potential avenues for future work.

\section{Lights Out and Linear Algebra}

In this section, we briefly review the connection between Lights Out and linear algebra. This review is covered in more detail in Section 2 of \cite{FM}. Additionally, we use this connection to establish some basic results.

\begin{definition}
Let $\Gamma$ be a graph whose vertices are labeled $1, 2, \ldots n$. The \emph{neighborhood matrix} $\A = [a_{ij}]$ of $\Gamma$ is the matrix with entries in the field $\mathbb{Z}_2$ where $a_{ij} = 1$ if $i=j$ or if vertices $i$ and $j$ are adjacent and $a_{ij} = 0$ otherwise.
\end{definition}

The neighborhood matrix gives a concrete method to determine whether Lights Out is universally solvable on $\Gamma$.

\begin{theorem}[\cite{GKT, GKTZ, Sutner1, Sutner2}]
Lights Out on a graph $\Gamma$ is universally solvable if and only if the neighborhood matrix of $\Gamma$ is invertible.
\label{thm:inv}
\end{theorem}

\begin{remark}
The invertibility of the neighborhood matrix does not depend on the labeling of the vertices of $\Gamma$. If $\A$ and $\A'$ are neighborhood matrices for $\Gamma$ given by two different labelings, the matrices are similar because conjugating $\A$ by a permutation matrix gives $\A'$. That is, $\A'$ is invertible if and only if $\A$ is invertible.
\end{remark}

Recall that any matrix with two identical rows is not invertible. In our context, the neighborhood matrix of $\Gamma$ has identical rows corresponding to distinct vertices $u$ and $v$ when $u$ and $v$ are adjacent to each other and adjacent to the same set of vertices in $V\Gamma \setminus \{u, v\}$. This observation gives the following results.

\begin{proposition}
Let $\Gamma$ be a graph with adjacent vertices $u$ and $v$. If $u$ and $v$ are adjacent to the same set of vertices in $V\Gamma \setminus \{u, v\}$, then Lights Out is not universally solvable on $\Gamma$.
\label{prop:sameadj}
\end{proposition}

Recall that $v \in V\Gamma$ is \emph{dominating} if $v$ is adjacent to every vertex in $V\Gamma \setminus \{ v \}$.

\begin{corollary}
If $\Gamma$ has more than one dominating vertex, then Lights Out is not universally solvable on $\Gamma$.
\label{cor:sameadj}
\end{corollary}

\begin{proof}
This is immediate from Proposition \ref{prop:sameadj} by setting $u$ and $v$ to be distinct dominating vertices.
\end{proof}

The proposition below, which is an immediate consequence of Corollary \ref{cor:sameadj}, was proved in a much more general setting by Keough and Parker \cite[Proposition 16]{KP}.

\begin{proposition}
If $\Gamma \in \Gne$ where $e > {N} - \lfloor \frac{n}{2} \rfloor$, then Lights Out is not universally solvable on $\Gamma$.
\label{prop:toomanyedges}
\end{proposition}

\begin{proof}
Note that $\Gamma$ has at least 2 dominating vertices, and this now follows from Corollary \ref{cor:sameadj}.
\end{proof}

\section{Graph Joins and Dominating Sets}

In this section, we study the interaction between graph joins and universal solvability. This was first studied, in the context of all parity realizable (APR) graphs, by Amin and Slater in \cite{AS1}. Our discussion is brief; for further details, see \cite{AS1}.

Given a graph $\Gamma$, a \emph{dominating set} of $\Gamma$ is a collection of vertices $S \subseteq V\Gamma$ so that each vertex of $\Gamma$ is adjacent to or equal to a vertex in $S$. We say that $S$ is an \emph{odd dominating set} if each $v \in V\Gamma$ is adjacent to or equal to an odd number of vertices in $S$, and we define \emph{even dominating set} analogously. More generally, assign each vertex in $\Gamma$ an element in $\{0, 1\}$. Now, we can ask if there exists a dominating set $S$ so that for each $v \in V\Gamma$ that has been assigned a 0, $v$ is adjacent to or equal to an even number of elements in $S$, and for each $v$ assigned 1, $v$ is adjacent to or equal to an odd number of elements in $S$. If this is true for each possible assignment of $0$'s and $1$'s, we say $\Gamma$ is \emph{APR}. Note that $\Gamma$ is APR if and only if Lights Out is universally solvable on $\Gamma$; here the assignment of $0$'s and $1$'s determines the initial configuration of Lights Out and the dominating set $S$ is the collection of vertices that must be pressed to win.

We apply this terminology in our study of graph joins. Specifically, if $\Gamma_1$ and $\Gamma_2$ are two graphs with disjoint vertex sets, the \emph{graph join} $\Gamma_1 + \Gamma_2$ is the graph with vertex set $V\Gamma_1 \cup V\Gamma_2$ where, along with all edges that are in $\Gamma_1$ and $\Gamma_2$ individually, there is an edge between each $u \in \Gamma_1$ and $v \in \Gamma_2$. For example, if $\Gamma_2$ is a single vertex, then $\Gamma_1 + \Gamma_2$ is the graph obtained by adjoining a dominating vertex to $\Gamma_1$. Amin and Slater examined the interaction between graph joins and APR graphs:

\begin{theorem}[Theorem 3 \cite{AS1}]
Let $\Gamma_1$ and $\Gamma_2$ be two graphs. Then $\Gamma_1 + \Gamma_2$ is APR if and only if
\begin{enumerate}
    \item both $\Gamma_1$ and $\Gamma_2$ are APR, and
    \item at least one of $\Gamma_1$ or $\Gamma_2$ has an odd dominating set of even cardinality.
\end{enumerate}
    \label{thm:as}
\end{theorem}

We will make use of Theorem \ref{thm:as} as we study graph joins in which one of the graphs has two or fewer vertices. We will start with the two vertex case and note that Corollary \ref{cor:sameadj} demonstrates that Lights Out is not universally solvable on $\Gamma_1 + \Gamma_2$ when $\Gamma_2$ has two vertices and one edge. The situation is slightly different when $\Gamma_2$ has two vertices $u$ and $v$ but no edges.

\begin{proposition}
Let $\Gamma$ be a graph and $\Gamma'$ be the graph with two vertices $u$ and $v$ and no edges. Then Lights Out is universally solvable on $\Gamma + \Gamma'$ if and only if Lights Out is universally solvable on $\Gamma$. 
\label{prop:isoedgedoesntomatter}
\end{proposition}

\begin{proof}
Note that Lights Out is universally solvable on $\Gamma'$ and that $V\Gamma'$ is an odd neighborhood dominating set for $\Gamma'$ with cardinality 2. Hence by Theorem \ref{thm:as}, Lights Out is universally solvable on $\Gamma + \Gamma'$ if and only if Lights Out is universally solvable on $\Gamma$.
\end{proof}

For the case of joining a general graph with a one vertex graph, we make use of Corollary 1 from \cite{Batal} which states that if Lights Out is universally solvable on $\Gamma$, then $\Gamma$ has an odd dominating set with the same parity as $|V\Gamma|$. Further, Lemma 1 of \cite{AS1} states that this is the only odd dominating set of $\Gamma$. That is, if Lights Out is universally solvable on $\Gamma$ and $|V\Gamma|$ is odd, then $\Gamma$ does not have an odd dominating set of even cardinality.

\begin{proposition}
Let $\Gamma$ be a graph with $n$ vertices where $n$ is odd and consider $\Gamma + \{v\}$. Lights Out is not universally solvable on $\Gamma$.
\label{prop:isovertodd}
\end{proposition}

\begin{proof}
Note that neither the one vertex graph $\{ v \}$ nor $\Gamma$ has an odd dominating set with even cardinality. Hence by Theorem \ref{thm:as}, Lights Out is not universally solvable on $\Gamma +\{v\}$.
\end{proof}

\begin{proposition}
Let $\Gamma$ be a graph with $n$ vertices where $n$ is even and consider $\Gamma + \{v\}$. Lights Out is universally solvable on $\Gamma + \{v\}$ if and only if Lights Out is universally solvable on $\Gamma$.
\label{prop:isoverteven}
\end{proposition}

\begin{proof}
If Lights Out is not universally solvable on $\Gamma$, then Theorem \ref{thm:as} implies that Lights Out is not universally solvable on $\Gamma + \{v\}$. Additionally, if Lights Out is universally solvable on $\Gamma$, then, by Corollary 1 in \cite{Batal}, $\Gamma$ has an odd dominating set of even cardinality. Since Lights Out is universally solvable on the graph with one vertex, Theorem \ref{thm:as} implies that Lights Out is universally solvable on $\Gamma+ \{v\}$.
\end{proof}

\section{Complement Graphs}

To simplify our arguments and figures, we will often work with complement graphs. Recall that for a graph $\Gamma$, the \emph{complement graph}, denoted $\Gamma^C$, is the graph on the same vertex set as $\Gamma$ in which there is an edge between two vertices in $\Gamma^C$ if and only if there is no edge between those vertices in $\Gamma$.  

To start, we will rephrase some of our results from Sections 2 and 3 in terms of complement graphs. Recall that the \emph{degree} of a vertex $v$, denoted $\deg(v)$, is the number of edges incident to $v$. A vertex in $\Gamma$ is {\it isolated}, i.e. $\deg(v) = 0,$ if and only if that vertex is dominating in $\Gamma^C$. Additionally, if $u$ and $v$ are non-adjacent vertices in $\Gamma^C$ that are adjacent to each vertex in $V\Gamma^C \setminus \{u,v\}$, then $u$ and $v$ are adjacent degree 1 vertices in $\Gamma$; that is, the edge between $u$ and $v$ in $\Gamma$ is an \emph{isolated edge}. 

Now we consider Propositions \ref{prop:sameadj}, \ref{prop:isoedgedoesntomatter}, \ref{prop:isovertodd}, and \ref{prop:isoverteven}, and Corollary \ref{cor:sameadj} in the context of complement graphs. Here, $\Gamma \setminus R$ denotes the graph obtained by removing $R$ and each edge incident to an element in $R$ from $\Gamma$, where $R \subseteq V\Gamma$.

\begin{proposition}
Let $v \in V\Gamma$.
\begin{enumerate}
\item If $v$ is an isolated vertex and $\Gamma$ has an even number of vertices, then Lights Out is not universally solvable on $\Gamma^C$.
\item If $v$ is an isolated vertex and $\Gamma$ has an odd number of vertices, then Lights Out is universally solvable on $\Gamma^C$ if and only if Lights Out is universally solvable on $(\Gamma \setminus \{v\})^C$.
\item If $v$ is adjacent to two or more vertices of degree 1, then Lights Out is not universally solvable on $\Gamma^C$.
\end{enumerate}
Let $u, v \in V\Gamma$ with $u \ne v$.
\begin{enumerate}
\setcounter{enumi}{3}
\item If $u$ and $v$ are isolated vertices in $\Gamma$, then Lights Out is not universally solvable on $\Gamma^C$.
\item Suppose there is an isolated edge between $u$ and $v$ in $\Gamma$.  Then Lights Out is universally solvable on $\Gamma^C$ if and only if Lights Out is universally solvable on $(\Gamma \setminus \{u,v\})^C$.
\end{enumerate}
\label{prop:compprops}
\end{proposition}

The proofs of the five statements in Proposition \ref{prop:compprops} are immediate from Proposition \ref{prop:isovertodd}, Proposition \ref{prop:isoverteven}, Proposition \ref{prop:sameadj}, 
Corollary \ref{cor:sameadj}, and Proposition \ref{prop:isoedgedoesntomatter} respectively.

To conclude this section, we prove two lemmas that will be used in example computations in Section 5.

\begin{lemma}
Suppose that $\Gamma$ has a vertex $v$ of degree $d \geq 3$. Then if $\Gamma$ has $d-1$ or fewer vertices of degree 2 or greater, then Lights Out is not universally solvable on $\Gamma^C$.
\label{lem:adjdegree}
\end{lemma}

\begin{proof}
Note that $v$ is adjacent to $d$ vertices in $\Gamma$ and at least two of those vertices have degree 1. By Proposition \ref{prop:compprops} (3), Lights Out is not universally solvable on $\Gamma^C$.
\end{proof}

\begin{lemma}
Suppose $\Gamma$ has a 4-cycle containing distinct non-adjacent degree 2 vertices. Then Lights Out is not universally solvable on $\Gamma^C$.
\label{prop:no4cycle}
\end{lemma}

\begin{proof}
Let $u$ and $v$ be the non-adjacent degree 2 vertices. Then, in $\Gamma^C$, $u$ and $v$ satisfy the hypotheses of Proposition \ref{prop:sameadj}.
\end{proof}

\section{Excess Degree}

In this section, we introduce the concept of excess degree and leverage this to understand the universal solvability of Lights Out on nearly complete graphs.

\begin{definition}
For a graph $\Gamma$, the \emph{excess degree} of $\Gamma$, denoted $\edeg(\Gamma)$, is given by
$$\edeg(\Gamma) = i - n + \displaystyle \sum_{v \in V\Gamma} \deg(v),$$
where $i$ is the number of isolated vertices of $\Gamma$ and $n = |V\Gamma|$.
\end{definition}

To understand excess degree, consider a graph $\Gamma$ with $\edeg(\Gamma) = 0$. Each non-isolated vertex in $\Gamma$ has degree 1, and so each connected component of $\Gamma$ is an isolated edge or an isolated vertex. For a graph $\Gamma$ with $m$ non-isolated vertices and $\edeg(\Gamma) = d$, the degrees of the non-isolated vertices of $\Gamma$ form an integer partition of $m+d$ into $m$ terms. That is, a graph with $\edeg(\Gamma) = d$ has, among the non-isolated vertices, $d$ more adjacencies than are minimally required. The extra adjacency could be spread out to at most $d$ vertices; in a graph with excess degree $d$, there are at most $d$ vertices with degree greater than 1. The other extreme case is for all of that extra adjacency to be concentrated in a single vertex; the maximum degree of a vertex in a graph of excess degree $d$ is $d+1$. Lastly, since $m+d$ is the sum of the degrees of the vertices of $\Gamma$, the integers $m$ and $d$ must have the same parity. 

\begin{definition}
Let $\mathcal{E}^n_d$ be the set of graphs $\Gamma$ with $n$ vertices and excess degree $d$ for which Lights Out is universally solvable on $\Gamma^C$. We denote $|\mathcal{E}^n_d| = E^n_d$.
\end{definition}

By Proposition \ref{prop:compprops} (4), each $\Gamma \in \mathcal{E}^n_d$ has no more than 1 isolated vertex. In particular, $\Gamma$ has no isolated vertices if and only if $n$ and $d$ have the same parity. The following proposition shows that when $\Gamma$ has an isolated vertex, $n$ is odd and $d$ is even.

\begin{proposition}
Let $n$ be even and $d$ be odd. Then $E^n_d = 0$.
\label{prop:edegparity}
\end{proposition}

\begin{proof}
To the contrary, suppose $\Gamma \in \mathcal{E}^n_d$ and note that $\Gamma$ has an isolated vertex. By Proposition \ref{prop:compprops} (1), Lights Out is not universally solvable on $\Gamma^C$, contradicting our assumption.
\end{proof}

After accounting for parity, the proposition below demonstrates that the value $E^n_d$ stabilizes as $n$ approaches $\infty$.

\begin{proposition}
Suppose $n$ and $d$ are positive integers with $n \geq 3d$. Then
$$E^n_d =
\begin{cases}
E^{3d}_{d} & \text{ if } \, n \, \text{ and } \, d \, \text{ have the same parity or if } \, n \, \text{ is odd,}\\[2 pt]
0 & \text{ otherwise.}
\end{cases}$$
\label{prop:excessdegntoinfty}
\end{proposition}

\begin{proof}
Let $\Gamma \in \mathcal{E}^n_d$. We will remove a collection of isolated edges and all isolated vertices from $\Gamma$ to create a graph $\Gamma' \in \mathcal{E}^{3d}_d$. This construction gives a one-to-one correspondence between the sets of graphs and, with Proposition \ref{prop:compprops} (2), establishes the result.

Note that by Proposition \ref{prop:compprops} (4), $\Gamma$ has 0 isolated vertices when $n$ and $d$ are of the same parity or 1 isolated vertex when $n$ and $d$ have opposite parities. Further, $\Gamma$ can have at most $d$ vertices with degree greater than 1. Consider the case when the parities of $n$ and $d$ agree, and note there are at least $n-d$ vertices of degree 1 and at most $2d$ of these can be adjacent to a vertex of degree greater than 1. That is, there are at least $n-3d$ vertices that are incident to an isolated edge, and hence at least $\frac{n-3d}{2}$ isolated edges. We remove $\frac{n-3d}{2}$ isolated edges from $\Gamma$ and call the resulting graph $\Gamma'$. Note that $\Gamma'$ has $3d$ vertices and excess degree $d$, and that $\Gamma'$ is independent of the specific $\frac{n-3d}{2}$ isolated edges chosen from $\Gamma$. By Proposition \ref{prop:compprops} (5), Lights Out is universally solvable on $\Gamma'^C$. That is, we have a well-defined map from $\mathcal{E}^n_d$ to $\mathcal{E}^{3d}_d$. This map is bijective; adding $\frac{n-3d}{2}$ isolated edges to any graph $\Gamma' \in \mathcal{E}^{3d}_d$ recovers $\Gamma$.

The same argument applies when $n$ and $d$ are of opposite parities. There are at least $n-1-d$ vertices of degree 1, and at most $2d$ of these can be adjacent to a vertex of degree greater than 1; there are at least $\frac{n-1-3d}{2}$ isolated edges. Removing those edges from $\Gamma$ gives $\Gamma' \in \mathcal{E}^{3d+1}_d$ as Lights Out is  universally solvable on $\Gamma'^C$ by Proposition \ref{prop:compprops} (5). Removing $\frac{n-1-3d}{2}$ isolated edges gives a one-to-one correspondence between $\mathcal{E}^n_d$ to $\mathcal{E}^{3d+1}_d$. Now, when $d$ is odd, Proposition \ref{prop:edegparity} shows that $E^{3d+1}_d = 0$.  When $d$ is even, Proposition \ref{prop:compprops} (2) shows that removing the isolated vertex $v$ gives $\Gamma \setminus \{v\} \in E^{3d}_{d}.$ This gives an invertible map from $\mathcal{E}^{3d+1}_{d}$ to $\mathcal{E}^{3d}_d$; adding an isolated vertex to a graph in $\mathcal{E}^{3d}_d$ is the inverse function. Hence, when $d$ is even and $n$ is odd, $E^n_d = E^{3d+1}_d = E^{3d}_d$.
\end{proof}

\begin{remark}
Proposition \ref{prop:excessdegntoinfty} extends to the cases when $n$ or $d = 0$.  When $n = 0$ it follows that $d$ is necessarily 0 and $\Gamma$ is the empty graph; we regard Lights Out as vacuously universally solvable on $\Gamma$. Hence $E^0_0 = 1$.

Rather than rely on the $n = 0$ case, suppose $\Gamma$ has $n > 0$ vertices and excess degree $d = 0$. Note that $\Gamma$ is a collection of isolated edges when $n$ is even and that same graph with an isolated vertex when $n$ is odd. When $n$ is even, removing all but one isolated edge results in Graph 1 of Figure \ref{fig:graphs}, while when $n$ is odd removing every isolated edge results in a graph with a single vertex. In both cases, the complement is a graph with no edges, and hence Lights Out is universally solvable on the complement and $\Gamma \in \mathcal{E}^n_0$. That is, $E^n_0 = 1$ for each non-negative integer $n$, extending Proposition \ref{prop:excessdegntoinfty}.
\label{rem:edeg0}
\end{remark}

\begin{remark}
Consider $\Gamma \in \mathcal{E}^{3d}_d$. This graph has 0 isolated vertices, and the degrees of the vertices of $\Gamma$ are a partition of $3d + d$ into $3d$ positive integers. Such integer partitions give a setwise partition of $\mathcal{E}^{3d}_d$.  In particular, let $\mathcal{Q}_d$ be the set of integer partitions of $4d$ into a sum of $3d$ positive integers, and let $Q_d = |\mathcal{Q}_d|$. We write $p \in \mathcal{Q}_d$ as $p = [p_1, p_2, \ldots, p_{4d}]$ where 
$$\displaystyle \sum_{j=1}^{4d} jp_j = 4d, \;\;\;\; \;\;\; \displaystyle \sum_{j=1}^{4d} p_j = 3d,$$ and $p_j$ is a non-negative integer for each $j$. Then let $\mathcal{S}_{p}$ be the set of graphs with $3d$ vertices and $p_j$ vertices of degree $j$ for each $j$ and let $S_p = |\mathcal{S}_p|$. Then $$E^{3d}_{d} \leq \displaystyle \sum_{p \in \mathcal{Q}_d} S_p.$$ Note that for each fixed $d$, the value $Q_d$ is finite and, for $p \in \mathcal{Q}_d$, the value $S_p$ is also finite. For a fixed value $d$, these observations together with Proposition \ref{prop:excessdegntoinfty} demonstrate that as $n$ approaches $\infty$, the value of $E^n_d$ is finite and depends only on the parity of $n$. 
\label{rem:Nndfinite}
\end{remark}

In light of Propositions \ref{prop:edegparity} and \ref{prop:excessdegntoinfty} and Remark \ref{rem:Nndfinite}, we will consider $E^{3d}_d$ when $1 \leq d \leq 5$. 

\begin{example}
Let $\Gamma \in \mathcal{E}^3_1$. Note that one vertex of $\Gamma$ has degree 2 while the other two have degree 1. By Proposition \ref{prop:compprops} (3), Lights Out is not universally solvable on $\Gamma^C$. Hence $\mathcal{E}^3_1 = \emptyset$ and $E^3_1 = 0$. Note that $E^4_1 = 0$ by Proposition \ref{prop:edegparity} and so $E^n_1 = 0$ for each $n \geq 3$.
\label{ex:edeg1}
\end{example}

\begin{example}
Let $\Gamma \in \mathcal{E}^6_2$. Observe that $\Gamma$ cannot have a vertex of degree 3 by Lemma \ref{lem:adjdegree}. In the case where every vertex of $\Gamma$ has degree 2 or less, we see that $\Gamma$ has exactly two vertices of degree 2. These vertices must be in the same connected component of $\Gamma$ by Proposition \ref{prop:compprops} (3). Hence $\Gamma$ is uniquely determined and is Graph 2 of Figure \ref{fig:graphs} and Lights Out is indeed universally solvable on $\Gamma^C$. Thus $E^n_2 = 1$ for all $n \geq 6$ by Proposition \ref{prop:excessdegntoinfty}.
\label{ex:edeg2}
\end{example}

\begin{example}
Let $\Gamma \in \mathcal{E}^{9}_3$. Note that $\Gamma$ cannot have a vertex of degree 3 or greater by Lemma \ref{lem:adjdegree}. It follows that every vertex of $\Gamma$ has degree 2 or less and $\Gamma$ has exactly three vertices of degree 2. All three of these degree 2 vertices are in the same connected component by Proposition \ref{prop:compprops} (3). This leaves Graphs 3 and 4 in Figure \ref{fig:graphs} as the only possibilities for $\Gamma$ and for each of these Lights Out is universally solvable on $\Gamma^C$. Hence $E^n_3 = 2$ if $n$ is odd and, by Proposition \ref{prop:edegparity}, $E^n_3 = 0$ if $n$ is even. 
\label{ex:edeg3}
\end{example}

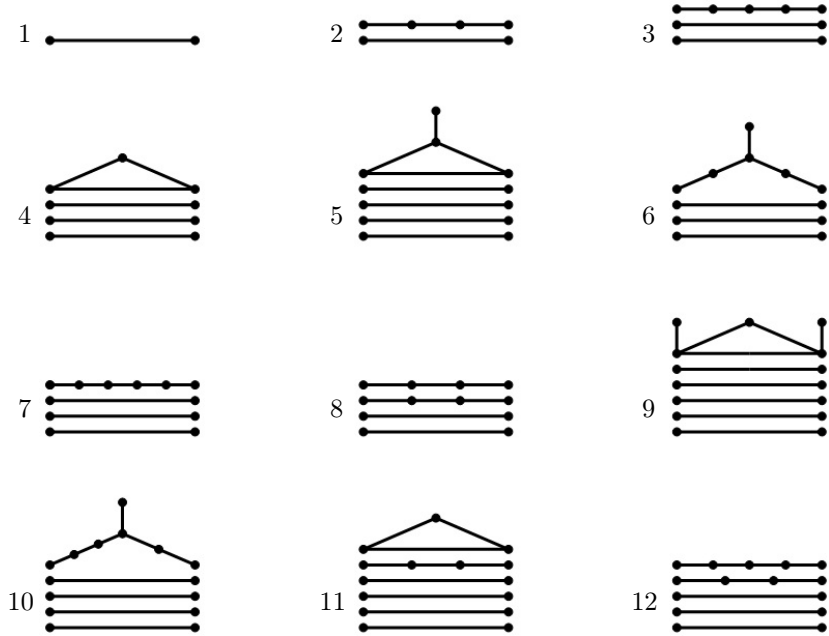
\begin{figure}[ht]
\centering
\begin{picture}(300, 250)
\put(10,0){\includegraphics[scale = 0.52]{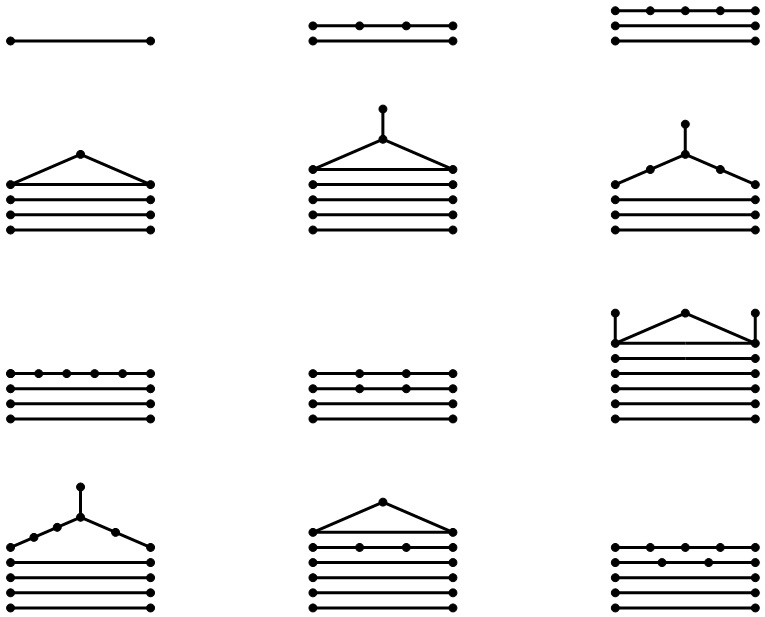}}

\put(2,225){1}
\put(120,225){2}
\put(238,225){3}

\put(2,156){4}
\put(120,156){5}
\put(238,156){6}

\put(2,83){7}
\put(120,83){8}
\put(238,83){9}

\put(-2,10){10}
\put(116,10){11}
\put(234,10){12}

\end{picture}
\caption{Graph 1 is the only graph in $\mathcal{E}^2_0$ while Graphs 2 through 12 are all of the graphs in $\mathcal{E}^{3d}_d$ where $1 \leq d \leq 5$. See Remark \ref{rem:edeg0} and Examples \ref{ex:edeg1}, \ref{ex:edeg2}, \ref{ex:edeg3}, \ref{ex:edeg4}, and \ref{ex:edeg5} for more discussion.}
\label{fig:graphs}
\end{figure}

\begin{example}
Let $\Gamma \in \mathcal{E}^{12}_4$. Note that $\Gamma$ cannot have a degree 4 vertex and cannot have two vertices of degree 3 by Lemma \ref{lem:adjdegree}. Now suppose $\Gamma$ has exactly one degree 3 vertex and all other vertices have degree less than 3. This means that $\Gamma$ has a vertex of degree 3 and two vertices of degree 2. In order for Lights Out to be universally solvable on $\Gamma^C$, the vertex of degree 3 must be adjacent to each of the vertices of degree 2; otherwise, Proposition \ref{prop:compprops} (3) applies. There are two graphs meeting this description, and they are Graphs 5 and 6 in Figure \ref{fig:graphs} and Lights Out is universally solvable on each corresponding $\Gamma^C$.

When $\Gamma$ has no vertices of degree 3 or greater, it must have exactly four vertices of degree 2. Each degree 2 vertex must be adjacent to another degree 2 vertex by Proposition \ref{prop:compprops} (3). There are two options; either all four degree 2 vertices are in the same connected component of $\Gamma$ or the degree 2 vertices reside in one of two connected components with each component containing exactly two degree 2 vertices. In that first case, note that the four vertices of degree 2 cannot form a 4-cycle by Lemma \ref{prop:no4cycle}. This limits the possible graphs in $\mathcal{E}^{12}_4$ with no vertices of degree 3 or greater to two possible graphs. These are Graphs 7 and 8 in Figure \ref{fig:graphs}. Lights Out is universally solvable on each corresponding $\Gamma^C$. Hence, $E^n_4 = 4$ for all $n \geq 12$ by Proposition \ref{prop:excessdegntoinfty}.
\label{ex:edeg4}
\end{example}

\begin{example}
Let $\Gamma \in \mathcal{E}^{15}_5$. Note that $\Gamma$ cannot have a vertex of degree 4 or greater by Lemma \ref{lem:adjdegree}. Now consider the case when $\Gamma$ has two vertices $u$ and $v$ of degree 3. Then $\Gamma$ has one vertex $w$ of degree 2. By Proposition \ref{prop:compprops} (3), each of $u$ and $v$ is adjacent to at most one vertex of degree 1. Hence $u$, $v$, and $w$ are pairwise adjacent. This uniquely determines $\Gamma$, which is Graph 9 of Figure \ref{fig:graphs}. Lights Out is universally solvable on $\Gamma^C$.

Now suppose $\Gamma$ has exactly one vertex $v$ of degree 3 and, in this case, $\Gamma$ has three vertices of degree 2. At least two of these degree 2 vertices are adjacent to $v$ by Proposition \ref{prop:compprops} (3). Additionally, no connected component of $\Gamma$ contains exactly one vertex of degree 2 and two vertices of degree 1 by Proposition \ref{prop:compprops} (3). It follows that all vertices of degrees 2 or 3 are in the same connected component of $\Gamma$. Note that the four vertices of degrees 2 and 3 cannot form a 4-cycle by Lemma \ref{prop:no4cycle}. Now, there are three possible graphs $\Gamma$ remaining in this case for which Lights Out could be universally solvable on $\Gamma^C$. Only one of these three complement graphs has an invertible neighborhood matrix; the graph $\Gamma$ corresponding to that complement graph is Graph 10 in Figure \ref{fig:graphs}.

It only remains to consider the case in which $\Gamma$ has no vertices of degree greater than 2, and so $\Gamma$ has exactly 5 vertices of degree 2. By Proposition \ref{prop:compprops} (3), there is no connected component of $\Gamma$ that contains exactly one vertex of degree 2. This means that either (i) all of the degree 2 vertices are in the same connected component of $\Gamma$ or (ii) three of the degree 2 vertices share a connected component of $\Gamma$ and the other two degree 2 vertices are together in another connected component of $\Gamma$. In case (i), there are two possible graphs $\Gamma$: a graph with a 5-cycle of degree 2 vertices and a graph with a connected component that is a length 6 edge path. For each of these graphs, Lights Out is not universally solvable on $\Gamma^C$. In case (ii), there are again two possible graphs $\Gamma$: a graph with a 3-cycle of degree 2 vertices and a graph with a connected component that is a length 4 edge path. For each of these graphs, Lights Out is universally solvable on $\Gamma^C$. These are Graphs 11 and 12 in Figure \ref{fig:graphs}. Hence $E^n_5 = 4$ if $n$ is odd and $n \geq 15$. Completing this discussion, $E^n_5 = 0$ when $n$ is even by Proposition \ref{prop:edegparity}.
\label{ex:edeg5}
\end{example}

\section{Results by number of edges}

In this section, we examine graphs with $n$ vertices and $N - \lfloor \frac{n}{2} \rfloor - m$ edges. Let $\mathcal{U}^n_m$ be the set of such graphs for which Lights Out is universally solvable and let $|\mathcal{U}^n_m| = U^n_m$.

\begin{proposition}
Let $m$ be a fixed positive integer and $n \geq 6m$. Then
$$U^n_m =
\begin{cases}
E^{6m}_{2m} & \text{ if } \, n \, \text{ is even, and}\\[2 pt]
E^{6m-3}_{2m-1} + E^{6m}_{2m} & \text{ if } \, n \, \text{ is odd.}
\end{cases}
$$
\label{prop:numofedges}
\end{proposition}

\begin{proof}
Let $\Gamma$ be a graph such that $\Gamma^C \in \mathcal{U}^n_m$. Then $\Gamma$ has $n$ vertices, $\lfloor \frac{n}{2} \rfloor + m$ edges, and either 0 or 1 isolated vertices. We consider cases based on the parity of $n$ and the number of isolated vertices.

When $n$ is even, $\Gamma$ has total degree $n + 2m$.  Now if $\Gamma$ has no isolated vertices, $\edeg(\Gamma) = 2m$ and $\Gamma \in \mathcal{E}^n_{2m}$. If $\Gamma$ has an isolated vertex, then $\edeg(\Gamma) = 2m+1$ is odd and no such $\Gamma$ exists by Proposition \ref{prop:edegparity}. Hence $U^n_m = E^n_{2m}$ when $n$ is even and Proposition \ref{prop:excessdegntoinfty} gives the result.

When $n$ is odd, $\Gamma$ has total degree $n-1 + 2m$. Now if $\Gamma$ has no isolated vertices, $\edeg(\Gamma) = 2m-1$ and $\Gamma \in \mathcal{E}^n_{2m-1}$. If $\Gamma$ has an isolated vertex, then $\edeg(\Gamma) = 2m$ and $\Gamma \in \mathcal{E}^n_{2m}$. Hence $U^n_m = E^n_{2m-1} + E^n_{2m}$ when $n$ is odd and Proposition \ref{prop:excessdegntoinfty} completes the proof.
\end{proof}

\begin{remark}
Proposition \ref{prop:numofedges} can be extended to the case when $m = 0$, as $U^n_0 = E^n_0$ for each $n$. By Remark \ref{rem:edeg0} and Example \ref{ex:edeg1}, $U^n_0 = 1$ for all positive integers $n$; that is, for each $n$ there is 1 graph with $n$ vertices and $N - \lfloor \frac{n}{2} \rfloor$ edges on which Lights Out is universally solvable.
\label{rem:u0case}
\end{remark}

Together with the results in Examples \ref{ex:edeg1}, \ref{ex:edeg2}, \ref{ex:edeg3}, \ref{ex:edeg4}, and \ref{ex:edeg5}, Proposition \ref{prop:numofedges} allows us to compute the number of graphs with $n$ vertices and $N - \lfloor \frac{n}{2} \rfloor - 1$ edges or $N - \lfloor \frac{n}{2} \rfloor - 2$ edges on which Lights Out is universally solvable. 

\begin{example}
For $n \geq 6$, if $n$ is even there is $E^6_2 = 1$ graph with $n$ vertices and $N - \lfloor \frac{n}{2} \rfloor - 1$ edges on which Lights Out is universally solvable. If $n$ is odd there is also only $E^{3}_1 + E^6_2 = 1$ such graph.  
\label{ex:mequals1}
\end{example}

\begin{example}
For $n \geq 12$, if $n$ is even there are $E^{12}_4 = 4$ graphs with $n$ vertices and $N - \lfloor \frac{n}{2} \rfloor - 2$ edges on which Lights Out is universally solvable. If $n$ is odd there are $E^{9}_3 + E^{12}_4 = 6$ such graphs.
\label{ex:mequals2}
\end{example}

\begin{remark}
The inequalities bounding $n$ in Examples \ref{ex:mequals1} and \ref{ex:mequals2} can be improved by considering the specific graphs in Figure \ref{fig:graphs} that are in $E^n_2$, $E^n_3$, $E^n_4$, and $E^n_5$.  In particular, $U^n_1 = 1$ for all $n \geq 4$ and if $n \geq 8$, we have $U^n_2=4$ when $n$ is even and $U^n_2=6$ when $n$ is odd.
\label{rem:ulowcases}
\end{remark}

\begin{remark}
Let $l$ be odd and let $\Gamma_l$ be the length $l$ edge path.  The adjacency matrix of $\Gamma_l^C$ is invertible and Lights Out is universally solvable on $\Gamma_l^C$. By Proposition \ref{prop:compprops} (2), adding a collection of isolated edges to $\Gamma_l$ maintains the universal solvability of Lights Out on the complement graph. Now let $m \geq 3$ and $n \geq 6m$. By Proposition \ref{prop:numofedges}, $U^n_m > 0$. To see this, note that the graph $\Gamma_{2m+1}$ with $2m-1$ isolated edges is in $\mathcal{E}^{6m}_{2m}$.
\label{rem:unonzero}
\end{remark}

\section{Probabilistic  Results}

Now we can state and prove our main theorem. Recall that $P_{n,e}$ is the probability that when a graph is chosen uniformly at random from $\mathcal{G}_{n,e}$, the set of graphs with $n$ vertices and $e$ edges, the resulting game of Lights Out is universally solvable.

\begin{theorem}
Suppose $m$ is a positive integer. There exists a positive integer $M$ so that for all $n > M$ and all $e$ with $$N - \left\lfloor \frac{n}{2} \right\rfloor - m \leq e \leq N,$$ the probability $P_{n,e}$ is maximized when $e = N - \lfloor \frac{n}{2} \rfloor.$  
\label{thm:main}
\end{theorem}

\begin{remark}
Our computational results, listed in Tables \ref{tab:8} to \ref{tab:13} in Section 8, give evidence supporting the conjectures that when $m=2$ and $3$ the respective values $M = 9$ and $11$ suffice.
\end{remark}

Before proving Theorem \ref{thm:main}, we make use of the following corollary to Theorem 5 from \cite{Wr}. We omit the proof, as the result is a direct algebraic consequence of Theorem 5 from \cite{Wr}.

\begin{corollary}
Fix a positive integer $m$. Let $k$ be any integer with $0 \leq k < m$. Then
$$\displaystyle \lim_{n \to \infty} \frac{G_{n,e_{k+1}}}{G_{n,e_k}} = \infty$$
where $G_{n,e} = |\mathcal{G}_{n,e}|$ is the number of graphs with $n$ vertices and $e$ edges and $e_k = \lfloor \frac{n}{2} \rfloor + k$.
\label{cor:wright}
\end{corollary}

Note that Theorem 5 of \cite{Wr} applies as we can choose $A > 0$ and $n$ sufficiently large so that 
$$\displaystyle \frac{2e_k}{n} - \log n < 1 + \frac{2m}{n} - \log n \leq \frac{-A}{\sqrt{n}}.$$ Now Theorem 5 of \cite{Wr} gives the following form for $G_{n,e_k}$ and an analogous form for $G_{n,e_{k+1}}$,
$$G_{n,e_k} = \frac{\sqrt{\pi}}{\exp(1)} \sqrt{\frac{2V}{1+\log V}} \: \frac{\binom{V(V-1)/2}{e_k}}{V!} \: (1 + \eta).$$
Here, $\eta$ is $O(n^{-C})$ for some constant $C > 0$, the value $V = \lfloor v \rfloor$ for the unique real number $v$ so that $v\ln v = 2e_k$, and $\exp(1) \approx 2.718$ is Euler's number.  Writing $G_{n,e_k}$ and $G_{n,e_{k+1}}$ in this form and taking the limit as $n$ approaches $\infty$ gives Corollary \ref{cor:wright}.

\begin{proof}[Proof of Theorem \ref{thm:main}]
By Proposition \ref{prop:toomanyedges}, note that the probability $P_{n,e} = 0$ for each $e$ with $N - \left\lfloor \frac{n}{2} \right\rfloor <  e \leq N$. Now assume $n \geq 6m$ and $0 \leq k < m$. Let $$K = \max_{0 \leq k < m} \left( \dfrac{U^n_{k+1}}{U^n_{k}}, \dfrac{U^{n+1}_{k+1}}{U^{n+1}_{k}} \right).$$
By Remarks \ref{rem:u0case}, \ref{rem:ulowcases}\, and \ref{rem:unonzero} each of $U^n_k$ and $U^{n+1}_k$ are positive and by Proposition \ref{prop:numofedges} and Remark \ref{rem:Nndfinite} the value $K$ is constant.  We write $ N - \left\lfloor \frac{n}{2} \right\rfloor - k = N - e_k$ and note that $P_{n,N-e_k} = U^n_k/G_{n,N-e_k}$. By Corollary \ref{cor:wright}, there exists $M$ so that if $n > M$ the inequality $$K < 
\displaystyle \frac{G_{n, e_{k+1}}}{G_{n, e_k}} = \displaystyle \frac{G_{n,N-e_{k+1}}}{G_{n,N-e_k}},$$
holds for each $k$. Now for all $n > M$, 
$$\displaystyle \frac{U^n_0}{G_{n,N-e_0}} > \displaystyle \frac{U^n_1}{G_{n,N-e_1}} > \ldots > \displaystyle \frac{U^n_m}{G_{n,N-e_m}} $$
and
$P_{n,e}$ is maximized when $e = N-e_0 = N - \left\lfloor \frac{n}{2} \right\rfloor$.
\end{proof}

\section{Monte Carlo Experiments}

In this section, we describe our computational verification of Theorem \ref{thm:main}. For each $n$ from 8 to 13 and each $e$ from 1 to $N - 1$, we performed 1,000,000 trials in which we selected a graph uniformly at random from $\mathcal{G}_{n,e}$. We neglected the $e=0$ and $e=N$ cases as $P_{n,0} = 1$ and $P_{n,N} = 0$ for each $n$. We then determined, for each chosen graph, if Lights Out is universally solvable on the graph by determining if the graph's neighborhood matrix is invertible. We accomplished this by implementing, in Java, a graph selection algorithm first described by Wormald in \cite{Wo}. More specifically, for each $e$ with $e \leq \lfloor N / 2 \rfloor$ we applied Wormald's algorithm directly, and for values of $e$ with $e \geq \lfloor N / 2 \rfloor + 1$ we selected a random graph with $N - e$ edges and used the complement of that graph. The results of these experiments are given in Tables \ref{tab:8} to \ref{tab:13}. The code for our implementation is available at \href{https://github.com/riyagoyal2134/LightsOut.git}{github.com/riyagoyal2134/LightsOut.git}. 

\begin{table}[ht]
\centering
\caption{For each $e$ from 1 to 27, the probability, approximated by 1,000,000 trials, that Lights Out is universally solvable on a graph chosen uniformly at random from the set of graphs with 8 vertices and $e$ edges. The margin of error at 95$\%$ confidence is less than .001.}
\begin{tabular}{cccccc}
\hline\\[-5 pt]

Number & Probability & Number & Probability & Number & Probability \\
of & Universally & of & Universally & of & Universally \\
Edges & Solvable & Edges & Solvable & Edges & Solvable \\[5 pt]
\hline\\[-5 pt]

1 & 0 & 10 & 0.343950 & 19 & 0.219285\\
2 & 0.499919 & 11 & 0.346355 & 20 & 0.167277\\
3 & 0.199316 & 12 & 0.359100 & 21 & 0.104666\\
4 & 0.363499 & 13 & 0.353405 & 22 & 0.071329\\
5 & 0.249911 & 14 & 0.354413 & 23 & 0.041872\\
6 & 0.321910 & 15 & 0.345785 & 24 & 0.091008\\
7 & 0.312796 & 16 & 0.336634 & 25 & 0 \\
8 & 0.321971 & 17 & 0.304621 & 26 & 0  \\
9 & 0.325698 & 18 & 0.273169 & 27 & 0\\
[5 pt]
\hline
\end{tabular}
\label{tab:8}
\end{table}

\begin{table}[ht]
\centering
\caption{For each $e$ from 1 to 35, the probability, approximated by 1,000,000 trials, that Lights Out is universally solvable on a graph chosen uniformly at random from the set of graphs with 9 vertices and $e$ edges. The margin of error at 95$\%$ confidence is less than .001.}
\begin{tabular}{cccccc}
\hline\\[-5 pt]

Number & Probability & Number & Probability & Number & Probability \\
of & Universally & of & Universally & of & Universally \\
Edges & Solvable & Edges & Solvable & Edges & Solvable \\[5 pt]
\hline\\[-5 pt]

1 & 0 & 13 & 0.359621 & 25 & 0.289836\\
2 & 0.499920 & 14 & 0.364539 & 26 & 0.264087\\
3 & 0.200248 & 15 & 0.363156 & 27 & 0.215122\\
4 & 0.363940 & 16 & 0.366063 & 28 & 0.180504\\
5 & 0.240663 & 17 & 0.364274 & 29 & 0.114246\\
6 & 0.302106  & 18 & 0.366417 & 30 & 0.095091\\
7 & 0.277388  & 19 & 0.362569 & 31 & 0.039722 \\
8 & 0.307850  & 20 & 0.365207 & 32 & 0.090537  \\
9 & 0.307524  & 21 & 0.360068 & 33 & 0\\
10 & 0.331077 & 22 & 0.353830 & 34 & 0\\
11 & 0.341577 & 23 & 0.336727 & 35 & 0 \\
12 & 0.354687 & 24 & 0.322319 &  & \\
[5 pt]
\hline
\end{tabular}
\label{tab:9}
\end{table}

\begin{table}[ht]
\centering
\caption{For each $e$ from 1 to 44, the probability, approximated by 1,000,000 trials, that Lights Out is universally solvable on a graph chosen uniformly at random from the set of graphs with 10 vertices and $e$ edges. The margin of error at 95$\%$ confidence is less than .001.}
\begin{tabular}{cccccc}
\hline\\[-5 pt]

Number & Probability & Number & Probability & Number & Probability \\
of & Universally & of & Universally & of & Universally \\
Edges & Solvable & Edges & Solvable & Edges & Solvable \\[5 pt]
\hline\\[-5 pt]

1 & 0 & 16 & 0.373795  & 31 & 0.265220\\
2 & 0.499400 & 17 & 0.378671 & 32 & 0.222100\\
3 & 0.199833 & 18 & 0.382819 & 33 & 0.173787\\
4 & 0.363571 & 19 & 0.386209 & 34 & 0.126419\\
5 & 0.230841 & 20 & 0.388778 & 35 & 0.082956\\
6 & 0.287559  & 21 & 0.389532 & 36 & 0.054569\\
7 & 0.254501  & 22 & 0.390373 & 37 & 0.032591 \\
8 & 0.279321  & 23 & 0.389508 & 38 & 0.024290  \\
9 & 0.285446  & 24 & 0.387827 & 39 & 0.015224\\
10 & 0.304450 & 25 & 0.382296 & 40 & 0.038469\\
11 & 0.320889 & 26 & 0.375407 & 41 & 0 \\
12 & 0.336026 & 27 & 0.365266 & 42 & 0\\
13 & 0.348091 & 28 & 0.350667 & 43 & 0\\
14 & 0.357128 & 29 & 0.328518 & 44 & 0\\
15 & 0.366398 & 30 & 0.300978 &  & \\
[5 pt]
\hline
\end{tabular}
\label{tab:10}
\end{table}

\begin{table}[ht]
\centering
\caption{For each $e$ from 1 to 54, the probability, approximated by 1,000,000 trials, that Lights Out is universally solvable on a graph chosen uniformly at random from the set of graphs with 11 vertices and $e$ edges. The margin of error at 95$\%$ confidence is less than .001.}
\begin{tabular}{cccccc}
\hline\\[-5 pt]
Number & Probability & Number & Probability & Number & Probability \\
of & Universally & of & Universally & of & Universally \\
Edges & Solvable & Edges & Solvable & Edges & Solvable \\[5 pt]
\hline\\[-5 pt]
1 & 0  & 19 & 0.388856  & 37 & 0.358598\\
2 & 0.499724  & 20 & 0.393536 & 38 & 0.344516\\
3 & 0.200551  & 21 & 0.396004 & 39 & 0.324121\\
4 & 0.363507  & 22 & 0.396182 & 40 & 0.299494\\
5 & 0.230845 & 23 & 0.397806 & 41 & 0.267545\\
6 & 0.283949  & 24 & 0.398648 & 42 & 0.229736\\
7 & 0.244384  & 25 & 0.398657 & 43 & 0.184645 \\
8 & 0.263545  & 26 & 0.397970 & 44 & 0.140461  \\
9 & 0.262488  & 27 & 0.397904 & 45 & 0.096906\\
10 & 0.283742 & 28 & 0.397958 & 46 & 0.068320\\
11 & 0.295869 & 29 & 0.398135 & 47 & 0.040787 \\
12 & 0.314470 & 30 & 0.397270 & 48 & 0.035003\\
13 & 0.328777 & 31 & 0.395805 & 49 & 0.014937\\
14 & 0.344268 & 32 & 0.393628 & 50 & 0.038504\\
15 & 0.357644 & 33 & 0.390987 & 51 & 0\\
16 & 0.368724 & 34 & 0.386070 & 52 & 0 \\
17 & 0.376659 & 35 & 0.378623 & 53 & 0\\
18 & 0.383929 & 36 & 0.370198 & 54 & 0 \\
[5 pt]
\hline
\end{tabular}
\label{tab:11}
\end{table}

\begin{table}[ht]
\centering
\caption{For each $e$ from 1 to 65, the probability, approximated by 1,000,000 trials, that Lights Out is universally solvable on a graph chosen uniformly at random from the set of graphs with 12 vertices and $e$ edges. The margin of error at 95$\%$ confidence is less than .001.}
\begin{tabular}{cccccc}
\hline\\[-5 pt]
Number & Probability & Number & Probability & Number & Probability \\
of & Universally & of & Universally & of & Universally \\
Edges & Solvable & Edges & Solvable & Edges & Solvable \\[5 pt]
\hline\\[-5 pt]
1 & 0  & 23 & 0.398679 & 45 & 0.338081\\
2 & 0.499463 & 24 & 0.400726 & 46 & 0.315321\\
3 & 0.200089 & 25 & 0.402630 & 47 & 0.288664\\
4 & 0.364492 & 26 & 0.404618 & 48 & 0.255777\\
5 & 0.230357 & 27 & 0.406422 & 49 & 0.217856\\
6 & 0.279425 & 28 & 0.406209 & 50 & 0.178374\\
7 & 0.240084 & 29 & 0.407586 & 51 & 0.136891\\
8 & 0.255557 & 30 & 0.407843 & 52 & 0.098461 \\
9 & 0.248556 & 31 & 0.407941 & 53 & 0.065641\\
10 & 0.264135 & 32 & 0.408232 & 54 & 0.041626\\
11 & 0.275422 & 33 & 0.407777 & 55 & 0.025507\\
12 & 0.290852 & 34 & 0.407194 & 56 & 0.016212\\
13 & 0.307977 & 35 & 0.408278 & 57 & 0.010773\\
14 & 0.323746 & 36 & 0.407095 & 58 & 0.008299\\
15 & 0.340155 & 37 & 0.404648 & 59 & 0.005670\\
16 & 0.354000 & 38 & 0.401894 & 60 & 0.014655 \\
17 & 0.364707 & 39 & 0.398517 & 61 & 0\\
18 & 0.375543 & 40 & 0.392940 & 62 & 0 \\
19 & 0.382263 & 41 & 0.387447 & 63 & 0\\
20 & 0.387929 & 42 & 0.379956 & 64 & 0 \\
21 & 0.392398 & 43 & 0.369140 & 65 & 0\\
22 & 0.396219 & 44 & 0.355199 &  &  \\
[5 pt]
\hline
\end{tabular}
\label{tab:12}
\end{table}

\begin{table}[ht]
\centering
\caption{For each $e$ from 1 to 77, the probability, approximated by 1,000,000 trials, that Lights Out is universally solvable on a graph chosen uniformly at random from the set of graphs with 13 vertices and $e$ edges. The margin of error at 95$\%$ confidence is less than .001.}
\begin{tabular}{cccccc}
\hline\\[-5 pt]
Number & Probability & Number & Probability & Number & Probability \\
of & Universally & of & Universally & of & Universally \\
Edges & Solvable & Edges & Solvable & Edges & Solvable \\[5 pt]
\hline\\[-5 pt]
1 & 0  & 27 & 0.406430 & 53 & 0.387123\\
2 & 0.499294 & 28 & 0.408858 & 54 & 0.380220\\
3 & 0.199568 & 29 & 0.409305 & 55 & 0.371224\\
4 & 0.363899 & 30 & 0.410125 & 56 & 0.358582\\
5 & 0.230106 & 31 & 0.411138 & 57 & 0.342884\\
6 & 0.279965 & 32 & 0.411297 & 58 & 0.324630\\
7 & 0.238693 & 33 & 0.412302 & 59 & 0.300909\\
8 & 0.252276 & 34 & 0.412295 & 60 & 0.272193\\
9 & 0.241871 & 35 & 0.412268 & 61 & 0.236771\\
10 & 0.253103 & 36 & 0.411811 & 62 & 0.198391\\
11 & 0.260253 & 37 & 0.411792 & 63 & 0.156529\\
12 & 0.272231 & 38 & 0.412040 & 64 & 0.116132\\
13 & 0.286001 & 39 & 0.412263 & 65 & 0.080770\\
14 & 0.302467 & 40 & 0.411530 & 66 & 0.053189\\
15 & 0.318041 & 41 & 0.411892 & 67 & 0.033189\\
16 & 0.333833 & 42 & 0.411205 & 68 & 0.022349\\
17 & 0.346637 & 43 & 0.411874 & 69 & 0.014105\\
18 & 0.358702 & 44 & 0.410961 & 70 & 0.012321 \\
19 & 0.368772 & 45 & 0.410177 & 71 & 0.005739\\
20 & 0.376524 & 46 & 0.410095 & 72 & 0.014552 \\
21 & 0.384010 & 47 & 0.407995 & 73 & 0\\
22 & 0.389572 & 48 & 0.405994 & 74 & 0 \\
23 & 0.394758 & 49 & 0.403351 & 75 & 0\\
24 & 0.398829 & 50 & 0.401896 & 76 & 0 \\
25 & 0.401359 & 51 & 0.397383 & 77 & 0\\
[5 pt]
\hline
\end{tabular}
\label{tab:13}
\end{table}

For completeness, we describe Wormald's algorithm below.  This algorithm is listed as ``Algorithm Random Orbit II'' and is applied to our situation in Section 4 of \cite{Wo}. Our discussion is brief; please see \cite{Wo} for further details.

Wormald's algorithm is based on an algorithm first described by Dixon and Wilf in \cite{Dixon} that chooses a graph uniformly at random from the set of graphs with $n$ vertices. Both algorithms make use of the action of the symmetric group $\sigma_n$ on the set of labeled graphs with $n$ vertices where the action is given by permuting the labels. Note that each unlabeled graph with $n$ vertices corresponds to an orbit of the action, and so the selection of an unlabeled graph reduces to choosing an orbit. Elements in each orbit occur $|\sigma_n|$ times among the ordered pairs $(\sigma, \Gamma)$ where $\sigma \in \sigma_n$ and $\Gamma$ is a labeled graph fixed by $\sigma$, we denote the set of such graphs by $\fix(\sigma)$. It follows that selecting an ordered pair uniformly at random selects an orbit uniformly at random.  Additionally, for each orbit $\mathcal{O}$ the cardinality $|\mathcal{O} \cap \fix(\sigma)|$ is constant as $\sigma$ ranges over a conjugacy class of $\sigma_n$. Given these observations, Dixon and Wilf's algorithm (1) selects a conjugacy class of $\sigma_n$, weighting the selection by the cardinality of each conjugacy class, together with a representative $\sigma$ of that conjugacy class; (2) selects a labeled graph $\Gamma$ uniformly at random from $\fix(\sigma)$; and (3) removes the labels from $\Gamma$. Naively, it is possible to pick a graph uniformly at random from $\mathcal{G}_{n,e}$ by performing Dixon and Wilf's algorithm repeatedly until a graph with $e$ edges is obtained. This naive approach is ineffective because there are relatively few graphs with $n$ vertices that have $e$ edges. Wormald's algorithm improves on the naive approach in two ways.

\begin{enumerate}
    \item Wormald's algorithm uses biased sampling to preferentially select graphs with $e$ edges; this is accomplished through the choice of $p_i$ and $q_i$ described later in this section. While this distribution makes it more likely that a graph with $e$ edges is selected, it remains uniform on the graphs with $e$ edges.
    \item Dixon and Wilf's algorithm begins by selecting a conjugacy class of $\sigma_n$. Weighting this choice requires, in the worst case, computing the number of elements of each conjugacy class of $\sigma_n$; there is a conjugacy class for each integer partition of $n$ making this computationally expensive for large $n$. Instead, Wormald's algorithm begins by selecting a subset $R_i$ of $\sigma_n$ where $R_i$ consists of permutations that leave $n-i$ elements fixed. Computing $|R_i|$ is not difficult and each trial requires this cardinality for only one value of $i$. 
\end{enumerate}

Selecting a subset $R_i$, rather than a conjugacy class, has the drawback that for an arbitrary orbit $\mathcal{O}$ the cardinality $|\mathcal{O} \cap \fix(\sigma)|$ is not constant as $\sigma$ ranges over $R_i$. Wormald's algorithm handles this inconsistency by weighting the selection of $R_i$ by a scaled overestimate for the maximum number, as $\sigma$ ranges over $R_i$, of ordered pairs $(\sigma, \Gamma)$ where $\Gamma \in \fix(\sigma)$ is a labeled graph with $e$ edges. This overestimate is denoted $B_i$ and, more specifically, we choose $B_i \geq |R_i|/P(\Gamma)$ where $P(\Gamma)$ is the probability that $\Gamma$ is chosen from $\fix(\sigma)$. Weighting by these overestimates simplifies the computational complexity of the process but introduces a failure state that is described later in this section.

The computation of the $B_i$'s uses values $p_i$ and $q_i$ which are obtained from the solution $r$ of the cubic equation
$$Nr(r^2 + 1) - e(r+1)(r^2 + 1) + \frac{i^2+2i-2ni}{2} \: r(r-1) = 0.$$
In our implementation, we approximate this solution by applying Newton's method.  Here, $p_i = \frac{r}{r+1}$ and $q_i = 1 - p_i$.  Now we determine the values $B_i$ by
$$B_0 = p_1^{-e}q_1^{e-N}, \; \text{ and } \; B_i = \left(\dfrac{n!}{(n-i)!}\right) p_i^{-e}q_i^{e-N} (p_i^2 + q_i^2)^{\frac{i^2+2i-2ni}{4}},$$
where $2 \leq i \leq n$. We then select an element $\sigma \in R_i$ uniformly at random. When $i=0$ we choose the identity permutation for $\sigma$ and when $i \ne 0$ we choose $\sigma$ by selecting a random permutation of $\{1,2,\ldots,i\}$ and accepting this choice provided that the selected permutation has no fixed points. We repeat this process if necessary; the expected number of trials needed to get a fixed point free permutation of $\{ 1, 2, \ldots, i\}$ approaches a value less than 3 as $n$ approaches $\infty$. Amending this permutation with the identity function on $\{ i+1, i+2, \ldots, n\}$ produces $\sigma$. We take $\{1,2, \ldots, i\}$ to be the first $i$ labels; conjugate elements of $\sigma_n$ will yield labeled graphs that differ only by permuting the labels and so will not affect the invertibility of the resulting neighborhood matrix.  

The action of $\sigma$ on $\{1, 2, \ldots, n\}$ induces an action on the 2-element subsets of $\{1, 2, \ldots, n\}$. Let $D$ be an orbit of 2-element subsets under this action. We select a collection of these orbits by including all of the edges represented by $D$ with probability $p_i^{|D|} / (p_i^{|D|} + q_i^{|D|})$. The 2-subsets that have been selected produce a labeled graph $\Gamma$, and if this graph has $e$ edges, it is accepted. Otherwise, we restart the algorithm. Note that the probability of choosing $\Gamma$ from the set of graphs fixed by $\sigma$ is
$$P(\Gamma) = \dfrac{p_i^e q_i^{N - e}}{\displaystyle \prod^n_{j=1} (p_i^j + q_i^j)^{h(\sigma,j)}}$$
where $h(\sigma,j)$ is the number of orbits $D$ of cardinality $j$. 

Now that a labeled graph with $e$ edges has been chosen, it is necessary to check if the failure state has occurred or if this choice constitutes a successful trial.  The algorithm fails with probability $1 - |R_i|/B_iP(\Gamma)$ and in this case, we restart the algorithm. The algorithm succeeds, returning $\Gamma$, with probability $|R_i|/B_iP(\Gamma)$. To complete our computation, we determine the invertibility of the neighborhood matrix of $\Gamma$.  

\begin{table}[ht]
\centering
\caption{For graphs with 8 to 13 vertices, the approximate runtime in hours (rounded to the nearest half hour) to complete 1,000,000 trials of our implementation of Wormald's algorithm.  Algorithm was implemented in Java on a MacBook Pro with an Apple M3 Pro chip, 18~GB of memory, running macOS Sequoia Version 15.5.}
\begin{tabular}{cccc}
\hline\\[-5 pt]
Number of & Approximate & Number of&  Approximate \\
Vertices & Runtime in Hours & Vertices & Runtime in Hours \\[5 pt]
8 & 4.5 & 11 & 54\\
9 & 12 & 12 & 114\\
10 & 30.5 & 13 & 198\\[5 pt]
\hline
\end{tabular}
\label{tab:runtime}
\end{table}

The runtime data for our implementation is given in Table \ref{tab:runtime}. Our implementation leverages the first of Wormald's improvements to the naive use of Dixon and Wilf's algorithm but not the second improvement as we only consider values of $n$ ranging from 8 to 13. It is possible that our runtime could be improved by selecting a conjugacy class of $\sigma_n$ and working with exact probabilities rather than selecting a set $R_i$ and including a failure state. We have not explored this potential improvement.

\section{Future Work}

There are multiple ways to  build on this work. One direction is to study variations of Lights Out that include more than two states for the lights. These versions have been explored by other researchers in \cite{BBS,EEJJMS,GP,GMT,HMP}. 

Our computational results suggest that $P_{n,e}$ has a local maximum when $e \approx N/2$.  Investigating this local maximum is another direction for future work. The precise determination of the value of this maximum, along with a rigorous proof of its existence, remains an open question.

Theorem \ref{thm:main} shows that when a graph is chosen uniformly at random from $\mathcal{G}_{n,e}$, the probability $P_{n,e}$ has a local maximum when $e = N - \lfloor \frac{n}{2} \rfloor$ for sufficiently large values of $n$. Furthermore, Theorem \ref{thm:main} states that for each $m$ there is a minimum number $f(m)$ so that for all $n$ with $n > f(m)$ the probability $P_{n,e}$ reaches its absolute maximum on the range of values of $e$ from $N - \lfloor \frac{n}{2} \rfloor - m$ to $N$
when $e = N - \lfloor \frac{n}{2} \rfloor$. Our computational results give conjectures for $f(1)$, $f(2)$, and $f(3)$. Determining the function that controls the relationship between the number of vertices and the breadth of edge values on which $e = N - \lfloor \frac{n}{2} \rfloor$ maximizes $P_{n,e}$ is an open and interesting question.

\bibliographystyle{plain}

\begin{thebibliography}{10}

\smallskip
\bibitem{AS1}
A.T. Amin and P.J. Slater, {``Neighborhood domination with parity restrictions in
graphs''}, \textit{Congr. Numer.} \textbf{91}(1992), 19–30. \MR{1208985}

\smallskip
\bibitem{AS2}
A.T. Amin and P.J. Slater, {``All parity realizable trees''}, \textit{J. Combin. Math. Combin. Comput.} \textbf{20} (1996), 53–63. \MR{1376697} Available at \href{https://combinatorialpress.com/article/jcmcc/Volume%20020/vol-20-paper%205.pdf}{Combinatorial Press}.

\smallskip
\bibitem{Anderson}
M. Anderson and T. Feil, \href{https://doi.org/10.1080/0025570X.1998.11996658}{``Turning \emph{Lights Out} with linear algebra''}, \textit{Math. Mag.} \textbf{71} (1998), 300-303. \MR{1573341}

\smallskip
\bibitem{BBS}
L. Ballard, E. Budge, and D. Stephenson, 
\href{https://doi.org/10.2140/involve.2019.12.181}{``\emph{Lights Out} for graphs related to one another by constructions''}, \textit{Involve} \textbf{12}:2 (2019), 181-201. \MR{3864213}

\smallskip
\bibitem{Batal}
A. Batal, \href{https://doi.org/10.31801/cfsuasmas.1051208} {``Parity of an odd dominating set''}, \textit{Commun. Fac. Sci. Univ. Ank. Ser. A1. Math. Stat.}, \textbf{71}:4 (2022), 1023–1028. \MR{4522662}

\smallskip
\bibitem{CM}
L. Chen and K. Mazur, \href{https://doi.org/10.46787/pump.v8i.4268} {``A recursive approach to a multi-state cylindrical \emph{Lights Out} game''}, \textit{The PUMP Journal of Undergraduate Research}, \textbf{8} (2025), 225–241. \MR{4917817}

\smallskip
\bibitem{CHKR}
R. Cowen, S.H. Hechler, J.W. Kennedy, and A. Ryba, {``Inversion and neighborhood
inversion in graphs''}, \textit{Graph Theory Notes N.Y.} \textbf{37} (1999), 37–41.
\MR{1725188}

\smallskip
\bibitem{Dixon}
J. D. Dixon and H. S. Wilf, \href{https://doi.org/10.1016/0196-6774(83)90021-4}{``The random selection of unlabeled graphs''}, \textit{J. Algorithms} \textbf{4} (1983), 205-213. \MR{0710717}

\smallskip
\bibitem{EEJJMS}
S. Edwards, V. Elandt, N. James, K. Johnson, Z. Mitchell, and D. Stephenson,
\href{https://doi.org/10.2140/involve.2010.3.17}{``\emph{Lights Out} on finite graphs''}, \textit{Involve} \textbf{3}:1 (2010), 17–32. \MR{2672500}

\smallskip
\bibitem{EES}
H. Eriksson, K. Eriksson, and J. Sj\"ostrand, \href{https://doi.org/10.1006/aama.2001.0739}{``Note on the lamp lighting problem''}, \textit{Adv. Appl. Math.} \textbf{27} (2004), 357–366. \MR{1868970}

\smallskip
\bibitem{FM}
B. Forrest and N. Manno, \href{https://doi.org/10.46787/pump.v5i0.2593}{``\emph{Lights Out} on a random graph''}, \textit{The Pump Journal of Undergraduate Research} \textbf{5} (2022), 165–175. \MR{4478559}

\smallskip
\bibitem{GP}
A. Giffen and D. B. Parker, {``On generalizing the \emph{Lights Out} game and a
generalization of parity domination''}, \textit{Ars Combin.} \textbf{111} (2013), 273–288.
\MR{3100179} Available at \href{https://combinatorialpress.com/article/ars/Volume%20111/volume-111-paper-23.pdf}{Combinatorial Press}.

\smallskip
\bibitem{GKT}
J.~Goldwasser,  W.~Klostermeyer, and G.~Trapp, \href{https://doi.org/10.1080/03081089708818520}{``Characterizing switch-setting problems''}, \textit{Linear Multilinear Algebra} \textbf{43}:1-3 (1997), 121–136.  \MR{1613183}  Available at \href{https://www.researchgate.net/publication/2384576_Setting_Switches_in_a_Grid}{Research Gate}.

\smallskip
\bibitem{GKTZ}
J.~Goldwasser, W.F.~Klostermeyer, G.E.~Trapp, G.E., and C.Q.~Zhang, ``Setting switches in a grid'', \textit{`` Technical Report TR-95-20}, Department of Statistics and Computer Science, West Virginia University (1995).

\smallskip
\bibitem{GMT}
S. Gravier, M. Mhalla, and E. Tannier, \href{https://doi.org/10.1016/S0304-3975(03)00285-8}{``On a modular domination game''},
\textit{Theoret. Comput. Sci.,} \textbf{306}:1-3 (2003), 291–303. \MR{2000178} 

\smallskip
\bibitem{HY}
Y. Hayata and M. Yamagishi,\href{https://doi.org/10.2140/involve.2019.12.713}{``On weight-one solvable configurations of \emph{Lights Out} puzzle''}, \textit{Involve}, \textbf{12}:4 (2019), 713-720. \MR{3941607}  

\smallskip
\bibitem{HMP}
M. Hunziker, A. Machiavelo, and J. Park,\href{https://doi.org/10.1016/j.tcs.2004.03.031}{``Chebyshev polynomials over finite fields and reversibility of $\sigma$-automata on square grids''}, \textit{Theoret. Comput. Sci.,} \textbf{320}:2-3 (2004), 465–483. \MR{2064312}

\smallskip
\bibitem{KP}
L. Keough and D. Parker. \href{https://doi.org/10.7151/dmgt.2481} {``An extremal problem for the neighborhood \emph{Lights Out} game''}, \textit{Discuss. Math. Graph Theory,} \textbf{44}:4 (2024), 997–1021. \MR{4740071}


\smallskip
\bibitem{AR}
A. Raghavan, ``On the rank of random, symmetric matrices over $\mathbb{Z}_{2}$ via random graphs'', \textit{Senior Thesis, Georgia Institute of Technology, Atlanta}, (2022, May 6).  Available at \href{https://repository.gatech.edu/server/api/core/bitstreams/4ebdf03a-f0e7-464f-affd-e2d1ca1785d9/content}{Georgia Institute of Technology}.

\smallskip
\bibitem{Sutner1}
K. Sutner, \href{https://doi.org/10.1007/BF03023823}{``Linear cellular automata and the Garden-of-Eden'',} \textit{Math. Intelligencer,} \textbf{11}:2 (1989), 49–53. \MR{0994964} 

\smallskip
\bibitem{Sutner2}
K. Sutner, \href{https://doi.org/10.2307/2323999}{``The $\sigma$-game and cellular automata''}, \textit{Amer. Math. Monthly,} \textbf{97}:1 (1990), 24–34. \MR{1034347}  

\smallskip
\bibitem{Wo}
N. C. Wormald, \href{https://doi.org/10.1137/0216048}{``Generating random unlabelled graphs''}, \textit{SIAM J. Comput.} \textbf{16}:4 (1987), 717-728. \MR{0899697}

\smallskip
\bibitem{Wr}
E.M. Wright, \href{https://doi.org/10.1112/plms/s3-28.4.577}{``Graphs on unlabelled nodes with a large number of edges''}, \textit{Proc. London Math. Soc. (3)} \textbf{28}:4 (1974), 577-594. \MR{0351922}

\smallskip
\bibitem{JY}
J.-M.H. Yates II, ``The \emph{Lights Out}! game: exploring solvability on graphs of order $N$'', \textit{Master’s Thesis, California State Polytechnic University, Pomona}, Spring 2023.  Available at \href{https://scholarworks.calstate.edu/downloads/8p58pm91w}{California State University System}.

\end{thebibliography}

\end{document}